%% file: omoti.tex
\documentclass[11pt,a4paper]{amsart}

\usepackage[dvipdfmx]{graphicx,color}
\usepackage[abs]{overpic} 
\usepackage{amsmath}
\usepackage{mathtools}
\usepackage{amssymb}
\usepackage{mathtools}
\usepackage{amsthm}
\usepackage{enumerate}
\usepackage[all]{xy}

\usepackage{amsfonts}

\title[Nonsmoothable actions of $\mathbb{Z}_2 \times \mathbb{Z}_2$ on Spin four-manifolds]{Nonsmoothable actions of $\mathbb{Z}_2 \times \mathbb{Z}_2$ on Spin four-manifolds}
\author{YUYA KATO}
\address{Graduate School of Mathematical Sciences, the university of Tokyo}
\email{omoti67@gmail.com}

\newcommand{\Rbb}{\mathbb{R}}

\renewcommand{\phi}{\varphi}
\renewcommand{\epsilon}{\varepsilon}
\renewcommand{\Psi}{\varPsi}
\renewcommand{\Phi}{\varPhi}

\newcommand{\Zbb}{\mathbb{Z}}

\theoremstyle{definition}
\newtheorem{df}{Definition}[section]

\theoremstyle{plain}
\newtheorem{thm}[df]{Theorem}
\newtheorem{prp}[df]{Proposition}
\newtheorem{lem}[df]{Lemma}
\newtheorem{cor}[df]{Corollary}

\theoremstyle{remark}
\newtheorem{rem}[df]{Remark}

\makeatletter

\@addtoreset{equation}{section}
\makeatother

\begin{document}
\maketitle
\begin{abstract}
We construct some nonsmoothable actions of $\mathbb{Z}_2 \times \mathbb{Z}_2$ on spin four-manifolds by using an equivariant version of Furuta's 10/8-inequality. The examples satisfy following property: any proper subgroup of $\mathbb{Z}_2 \times \mathbb{Z}_2$ is smoothable for some smooth structure.
\end{abstract}
\tableofcontents
\section{Introduction}

In 4-dimentional topology, it is well known that there are many differences between the smooth category and the topological category since Donaldson introduced gauge theory in 1980's. It is also the case with group actions on 4-dimentional manifolds. For instance some applications of the Yang-Mills theory to group actions are found in \cite{F2,NH,R1} and those of the Seiberg-Witten theory in \cite{Br,Bo,K,N1,R2,U} (A good general reference is Edomonds' survey \cite{E}).

In particular K. Kiyono constructed non-smoothable group actions on most of closed spin topological 4-manifolds for cyclic groups with sufficiently large arbitrary prime orders. Kiyono made use of the construction of locally-linear cyclic group actions due to Ewing-Edmonds.

N. Nakamura constructed a subtle example : a locally-linear action of $\Zbb \times \Zbb$ on a 4-manifold, which is homeomorphic to the connected sum of Enriques surface and $S^2 \times S^2$, such that the actions of the subgroups $\Zbb \times \{ 0 \}$ and $\{ 0 \} \times \Zbb$ are smooth for some smooth structure on the manifold respectively while the action of the whole group $\Zbb \times \Zbb$ is not smooth for any smooth structure.

In this paper, we construct some action of $\Zbb_2 \times \Zbb_2$ on a family of spin 4-manifolds which satisfy a similar property.
\begin{thm}
For positive integers $l_1$, $l_2$, $k$ with $l_1 \geq 3k$, $l_2 \geq 3k$, $k \geq 1$, there is a locally linear action of $\Zbb_2 \times \Zbb_2$ on $X = (2l_1 + 2l_2 + 1 - 6k)(S^2 \times S^2) \sharp 4k(K3)$ such that the action of each proper subgroup is smooth for some smooth structure but the action of the whole group is not smooth for any smooth structure.
\end{thm}
Our construction is rather elementary in the sense that we do not use Ewing-Edmonds' construction.

To show the nonsmoothability part of the above theorem, we apply an equivariant version of 10/8-theorem which we prove using the Seiberg-Witten theory. J. Bryan \cite{Br} and C. Bohr \cite{Bo} already used some equivariant versions of the 10/8-theorem \cite{F1} to obtain some results on group actions. They used the group actions preserving spin${}^c$ structure. In our equivariant setting, the group action is not assume to preserve spin${}^c$ structures. Nakamura gave a slight generalization of the Seiberg-Witten theory in  \cite{N2,N3} and his generalization is constructed for a 4-dimentional manifold with a free involution which is not necessarily preserving spin${}^c$ structure. We generalize Nakamura's formulation for involution which is not necessarily free.

The organization of this paper is as follows. In Section 2 we state the equivariant version of the 10/8-theorem. In Section 3 we construct the examples of nonsmoothable actions. In Sections 4 we recall some basic properties of the Seiberg-Witten equations and a class of involution which is called odd involution. In Section 5 we prove our main theorems.

{\bf Acknowledgment}
I would like to thank my adviser Mikio Furuta for his helpful advise and powerful encouragement. This work is supported by the Program for Leading Graduate Schools.

\section{Main Theorems}
Let $X$ be a smooth oriented closed 4-manifold and $Q$ the intersection form which is defined on $ H^2 (X;\mathbb{R}) $. Let $\sigma(X)$ be the signature of $Q$. We assume $\sigma(X) \leq 0$. Suppose $ \iota $ is a smooth involution on $ X $. Let $I$ be the involutive homomorphism on the cohomology groups of X defined by
\begin{align*}
I := -{\iota}^{*} : H^{k}(X;\mathbb{R}) \to H^{k}(X;\mathbb{R})
\end{align*} 
for $ k = 0, 1, 2, \cdots  $.

Let $ H^{k}(X;\mathbb{R})^I $ be the $I$-invariant part of $ H^{k}(X;\mathbb{R}) $ and $ Q^I $ be the restriction of $Q$ to $ H^2 (X;\mathbb{R})^I $. We denote by $ b_{+}^I (X) $ the maximal dimension of positive definite subspaces for $ Q^I $.

We use the following definition.

\begin{df}
	Let $ X $ be a smooth spin 4-manifold, $ P \to X $ be the principal $Spin(4)$-bundle for the spin structure and $ \iota : X \to X $ be a smooth involution on $ X $ preserving the spin structure. We call that $ \iota $ is {\it odd} if $ \tilde{\iota} : P \to P $, a lift of $ \iota $, satisfies $ \tilde{\iota} \circ \tilde{\iota} = -1 $ and {\it even} if $ \tilde{\iota} \circ \tilde{\iota} = 1 $.
\end{df}

In the above definition, if $X$ has two smooth structures for both of which $\iota$ is smooth, then the parity of $\iota$ does not depend on the smooth structures.

The following lemma is well known. See Lemma $\ref{lem1}$ for an explanation.

\begin{lem}
	Let $ X $ be a smooth spin 4-manifold and $ \iota $ be a smooth involution on $ X $ preserving its spin structure. Assume that the fixed point set $ X^{\iota} $ of $ \iota $ is non-empty. If $ X^{\iota} $ has a 2-dimensional component then $ \iota $ is odd, and if $ X^{\iota} $ has a 0-dimensional component then $ \iota $ is even.
\end{lem}

One of our main theorems is:

\begin{thm}\label{thm1}
	Let $ ( X , \mathfrak{s} ) $ be a smooth closed oriented connected spin 4-manifold with spin structure $ \mathfrak{s} $ and $ \iota : X \to X $ be a smooth odd involution on X preserving the spin structure $ \mathfrak{s} $. Suppose $\sigma(X) \leq 0$. Then the following inequality holds
	\begin{align*}
	{b_{+}^I}(X) \geq -\frac{\sigma (X)}{16}.
	\end{align*}

\end{thm}

Suppose we have two involutions $ \iota_1 $ and $ \iota_2 $ on X which are mutually commutative. Then they give an action of $\Zbb_2 \times \Zbb_2$. They also give the involutive homomorphisms $I_1$ and $I_2$ on the cohomology groups of $X$ defined by
\begin{align*}
I_j := -\iota_j : H^k(X;\Rbb) \to H^k(X;\Rbb)
\end{align*}
for $j=1,2$ and $k=0,1,2,\cdots$.
Let $H^k(X;\Rbb)^{\langle I_1, I_2 \rangle}$ be the $\langle I_1, I_2 \rangle$-inariant part of $H^k(X;\Rbb)$ and $Q^{\langle I_1, I_2 \rangle}$ be the restriction of $Q$ to $H^2(X;\Rbb)^{\langle I_1, I_2 \rangle}$. We denote by $b^{\langle I_1, I_2 \rangle}_+(X)$ the maximal dimension of positive definite subspace for $Q^{\langle I_1, I_2 \rangle}$.

Our second main theorem is:

\begin{thm}\label{thm2}
	Let $ (X,\mathfrak{s}) $ be a smooth closed oriented connected spin 4-manifold with spin structure $ \mathfrak{s} $ and $ \iota_1 $ and $ \iota_2 $ be two commutative odd involutions preserving the spin structure $ \mathfrak{s}  $ such that their composition $ \iota_1 \circ \iota_2 = \iota_2 \circ \iota_1 $ is an even involution. Suppose $\sigma(X) \leq 0$. Then the following inequality holds
	\begin{align*}
	b_{+}^{\langle I_1, I_2 \rangle} (X) \geq -\frac{\sigma (X)}{32} + \frac{| \mathrm{index}_{I_1 \circ I_2} D | }{8}.
	\end{align*}
\end{thm}

Here $D$ is the Dirac operator and $I_1 \circ I_2$ is a lift of $\iota_1 \circ \iota_2$ on the positive and negative spinor bundles which will be explained in Subsection $\ref{Odd}$, and we define
\begin{align*}
	\mathrm{index}_{I_1 \circ I_2} D = \mathrm{trace}_{I_1 \circ I_2} (\mathrm{Ker}D) - \mathrm{trace}_{I_1 \circ I_2} (\mathrm{Coker}D),
\end{align*}
where the traces are defined to be those for real vector spaces.

In particular when $\iota_1 = \iota_2$, since $I_1 \circ I_2 = \pm \mathrm{id}$, we have $|\mathrm{index}_{I_1 \circ I_2} D| = |\mathrm{index} D| = - \frac{\sigma(X)}{4}$. 

When $\iota_1$ is not equal to $\iota_2$, the fixed point set $X^{\iota_1 \circ \iota_2}$ is isolated. For a fixed point $x \in X^{\iota_1 \circ \iota_2}$, the action of $I_1 \circ I_2$ on the fiber of the positive spinor bundle at $x$ is either $+\mathrm{id}$ or $-\mathrm{id}$. We call $x$ a positive fixed point if it is $+\mathrm{id}$, and negative fixed point if it is $-\mathrm{id}$. For a positive fixed point $x$, the action of $I_1 \circ I_2$ on the negative spinor bundle at $x$ is $-\mathrm{id}$. For a negative fixed point, it is $+\mathrm{id}$. Let $n_+$ be the number of the positive fixed points, and $n_-$ the number of negative ones. The Lefschetz formula for the equivariant index of elliptic operator \cite{AS} implies:

\begin{prp}\label{prp1}
When $\iota_1$ is not equal to $\iota_2$
	\begin{align*}
		\mathrm{index}_{I_1 \circ I_2} D = \frac{n_+ - n_-}{2}
	\end{align*}
\end{prp}

\section{Examples of nonsmoothable actions}
In this section we explain how to construct nonsmoothable actions, assuming our main theorems. We have two types of examples of the nonsmoothable actions. The first one is an example of nonsmoothable $ \mathbb{Z}_2 $-action. The second one may be more interesting than the first one : it is an example of nonsmoothable $ \mathbb{Z}_2 \times \mathbb{Z}_2 $-action such that each of the restricted action to any proper subgroup is smoothable for some smooth structure. 

\subsection{$ \mathbb{Z}_2 $-actions}

Let $\iota_0$ be a smooth diffeomorphism on $ S^2 \times S^2 $ defined by

\begin{align*}
\iota_0 : S^2 \times S^2 \ni \left(\begin{pmatrix} x \\ y \\ z \end{pmatrix},\begin{pmatrix} x' \\ y' \\ z' \end{pmatrix} \right) \mapsto \left(\begin{pmatrix} -x \\ -y \\ z \end{pmatrix},\begin{pmatrix} x' \\ y' \\ z' \end{pmatrix} \right) \in S^2 \times S^2.
\end{align*}
The fixed point set of $ \iota_0 $ is $ \{ (0,0,1), (0,0,-1) \} \times S^2 $, which implies $\iota_0$ is an odd involution. See Figure 1. This simple $\Zbb_2$-action plays a role of building block in construction of our nonsmoothable actions. 
\begin{figure}[h]
	\def \svgwidth{0.3\columnwidth}
	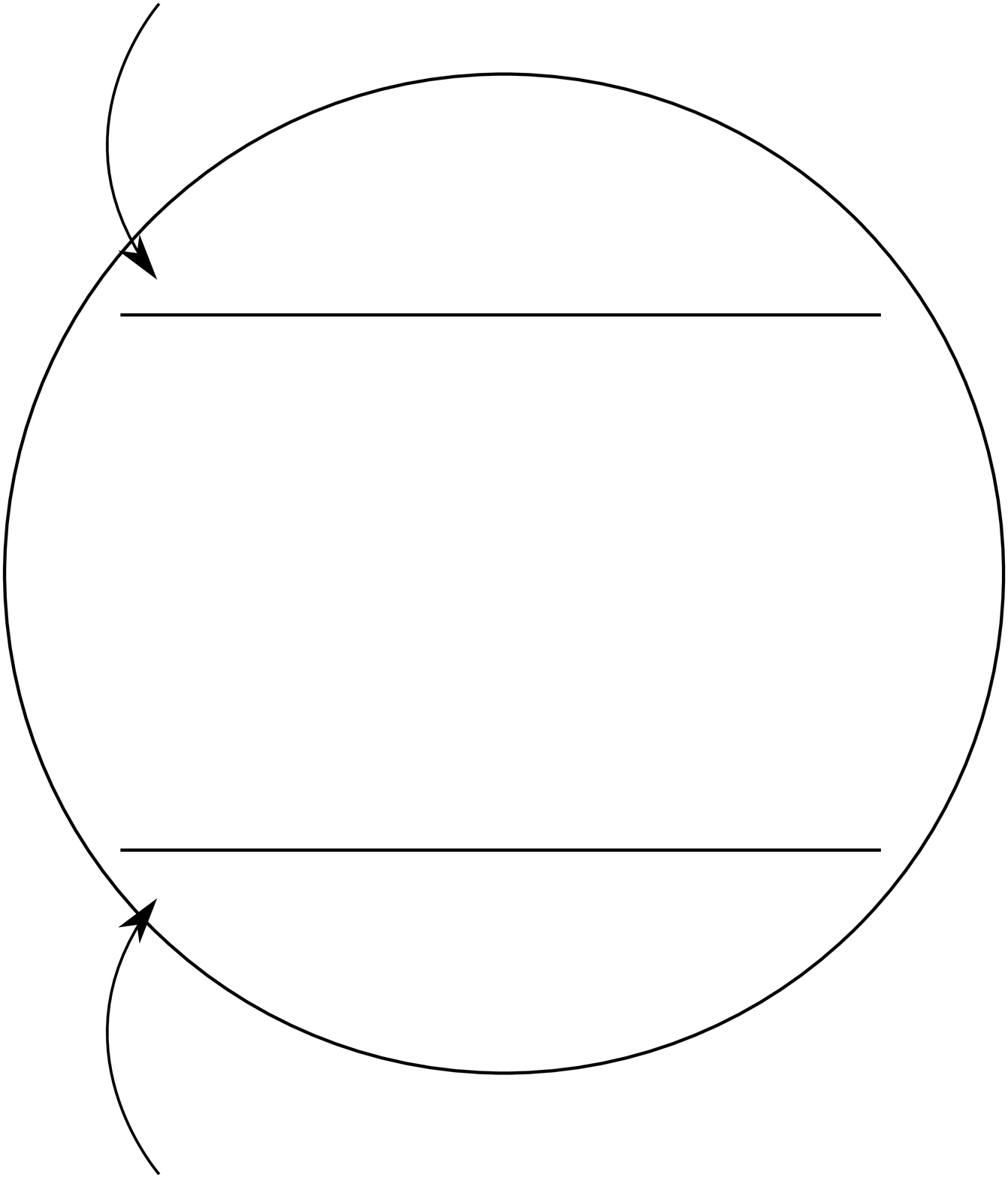
	\caption{}
\end{figure}

Next we take a connected sum of several copies of $ \iota_0 $'s in the following way. Around a fixed point, $ \iota_0 $ locally looks like the following involution $ i_0 $ on $ \mathbb{R}^4 $.
\begin{align*}
i_0 : \mathbb{R}^4 \ni \begin{pmatrix} x \\ y \\ z \\ w \end{pmatrix} \mapsto \begin{pmatrix} -x \\ -y \\ z \\ w \end{pmatrix} \in \mathbb{R}^4
\end{align*}
Now we consider taking a connected sum of two copies of $ i_0 $ at the origins. Let $ i'_0 $ be the restriction of $ i_0 $ to the complement of the open unit ball in $\mathbb{R}^4$. Take two copies of $ i'_0 $. These two involutions on $ A := \{ x \in \mathbb{R}^4 |  \| x \| \geq 1 \} $ agree on boundary $ \partial A = S^3 $, so we can define an involution on $ A \cup_{S^3} A $ such that equal to $ i_0 $ on two $ A $'s. We denote it by $i_0 \sharp i_0$. See Figure 2.

\begin{figure}[h]
	\def \svgwidth{0.7\columnwidth}
	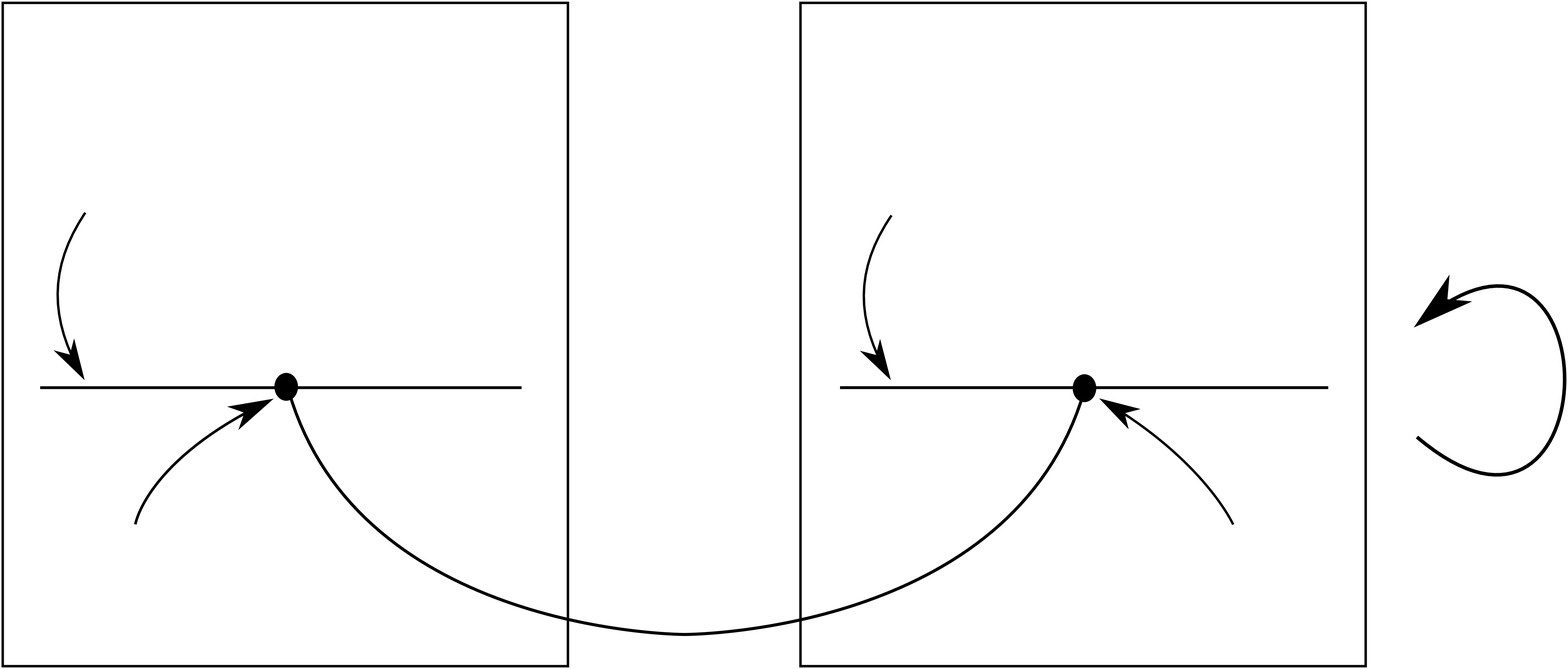
	\caption{}
\end{figure}

Similar to the above construction, we take a connected sum of two $ \iota_0 $'s at fixed points. As in the previous construction we further take a connected sum of $ l $-times copies of $ \iota_0 $'s, which we denote by $ \sharp^l \iota_0 $. See Figure 3.

\begin{figure}[h]
	\def \svgwidth{\columnwidth}
	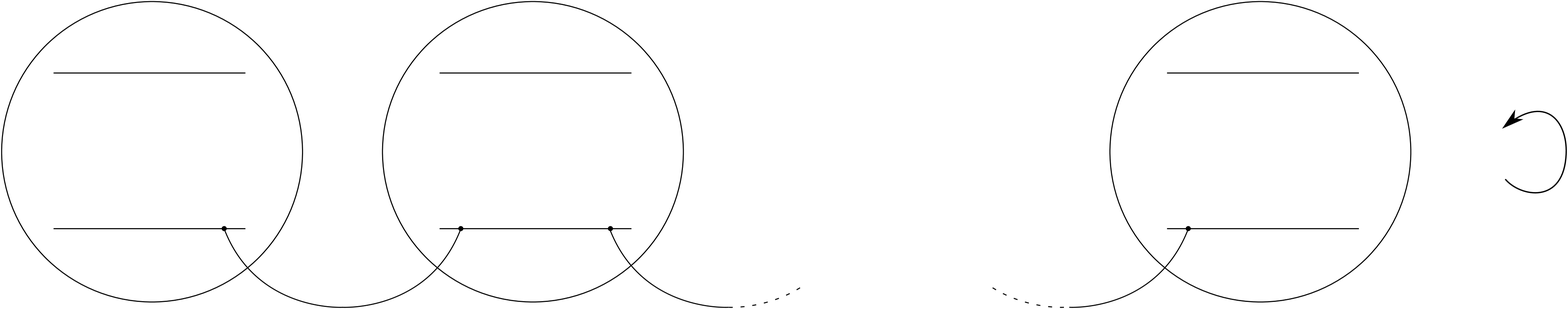
	\caption{}
\end{figure}

The last step of the construction is as follows. First, we define an topological involution $ \kappa $ on $ \sharp^k (-E_8) \coprod \sharp^k (-E_8) $ as exchange of their components. 

Second, we take the connected sum of $ \kappa $ and $ \sharp^l \iota_0 $. Take a point $ p \in \sharp^l (S^2 \times S^2) $ outside the fixed point set of $ \sharp^l \iota_0 $ and remove small open balls $ B_p $ and $ B'_p $ around $ p $ and $ \sharp^l \iota (p) $. Similarly take a some point $ q $ in $ \sharp^k (-E_8) \coprod \sharp^k (-E_8) $ and remove small open balls $ B_q $ and $ B'_q $ around $ q $ and $ \kappa (q) $.

Now we have two involutions on $ S^3 \coprod S^3 \cong \partial ( \sharp^l(S^2 \times S^2) \setminus B_p \cup B'_p) \cong \partial ( \sharp^k (-E_8) \coprod \sharp^k (-E_8) \setminus B_q \cup B'_q ) $, where $ \cong $ stands for the homeomorphisms, which are obtained as restrictions of $ \sharp^l \iota_0 $ and $ \kappa $ to the boundaries of their domains. These two coincide as involutions on $ S^3 \coprod S^3 $ because these are both exchanging the components. So we obtain an involution $ \iota $ on $ \sharp^l(S^2 \times S^2) \sharp (\sharp^{2k} (-E_8)) $ which is the connected sum of $ \sharp^l \iota_0 $ and $ \kappa $. See Figure 4.

\begin{figure}[h]
	\def \svgwidth{\columnwidth}
	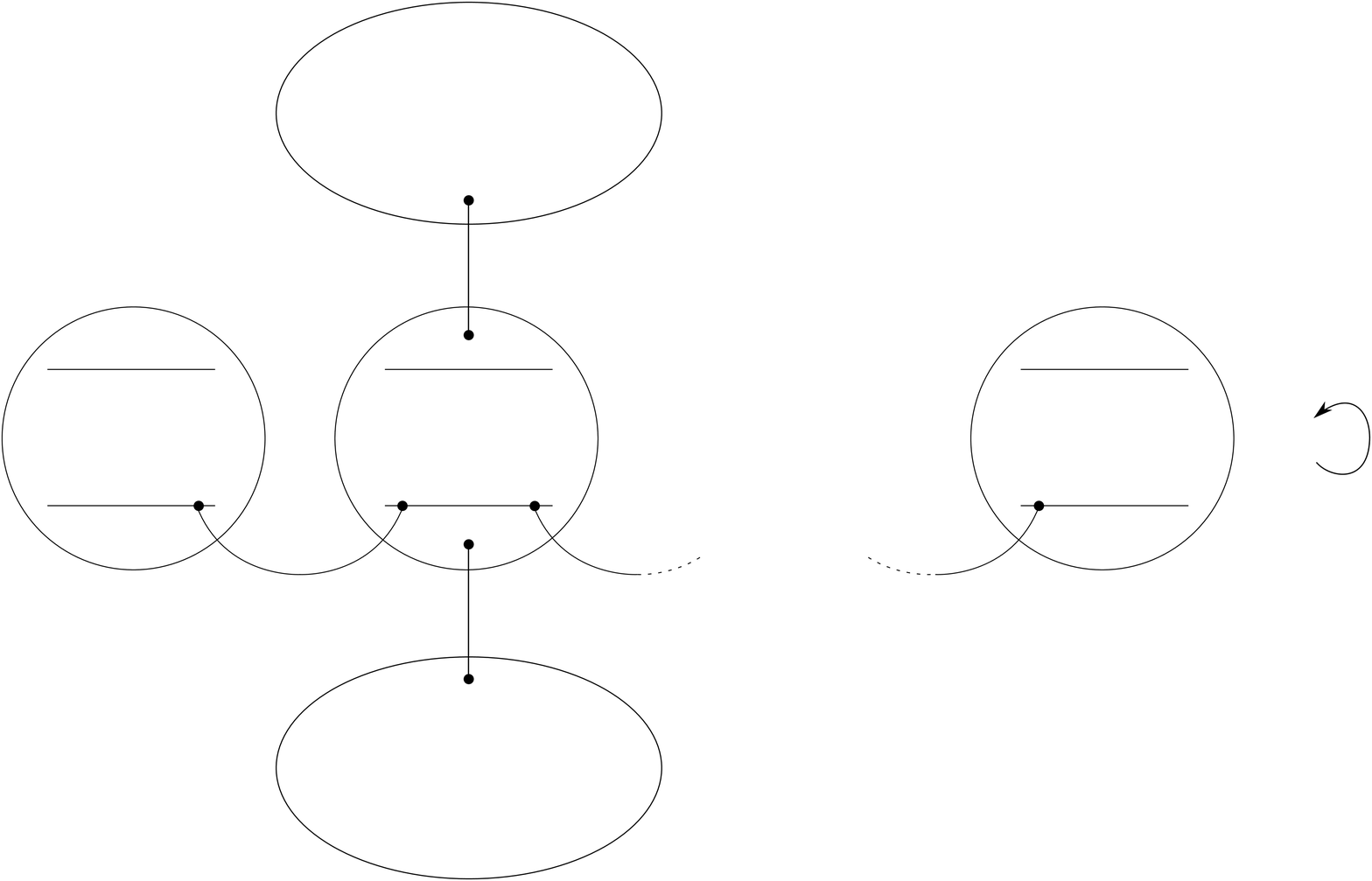
	\caption{}
\end{figure}

Now we show that $ \iota $ is a nonsmoothable $\mathbb{Z}_2$-action when $ k $ is grater than or equal to 1. Set $ X := \sharp^l (S^2 \times S^2) \sharp (\sharp^{2k} (-E_8)) $. Notice that $ X $ has a smooth structure if $ l \geq 3k $. 

\begin{prp}
	$ \iota : X \to X $ is nonsmoothable when $ k \geq 1 $.
\end{prp}
\begin{proof}
	Suppose that $ \iota $ is smoothable and consider the smooth structure on $ X $ for which $ \iota $ is smooth. Since $ X $ is simply connected it has a unique spin structure $\mathfrak{s}$. Since $\iota$ is smooth, odd, and preserving $ \mathfrak{s} $, Theorem $\ref{thm1}$ implies the inequality
	\begin{align*}
	b_{+}^I (X) \geq - \frac{1}{16} \sigma (X).
	\end{align*}
On the other hand a simple calculation shows $b_{+}^I(X) = 0$ and $- \frac{1}{16} \sigma (X) = k \geq 1$, which is a contradiction.
\end{proof}

\subsection{$ \mathbb{Z}_2 \times \mathbb{Z}_2 $-actions}

The construction of the second example is similar to the first one.

Define another smooth involution $ \iota_0' $ on $ S^2 \times S^2 $ by
\begin{align*}
\iota_0' : S^2 \times S^2 \ni \left(\begin{pmatrix} x \\ y \\ z \end{pmatrix},\begin{pmatrix} x' \\ y' \\ z' \end{pmatrix} \right) \mapsto \left(\begin{pmatrix} x \\ y \\ z \end{pmatrix},\begin{pmatrix} -x' \\ -y' \\ z' \end{pmatrix} \right) \in S^2 \times S^2.
\end{align*} 
The fixed point set of $ \iota'_0 $ is $ S^2 \times \{ (0,0,1) , (0,0,-1) \} $.

We first take connected sum of five copies of $ \iota_0 $ on five copies of $ S^2 \times S^2 $. We denote these five $ S^2 \times S^2 $'s equipped with the action of $ \iota_0 $ by $ S_0 $, $ S_1 $, $ S_2 $, $ S_3 $ and $ S_4 $ respectively. We connect $ S_0 $, $ S_1 $ and $ S_2 $ at fixed points of $ \iota_0 $, and $ S_0 $, $ S_3 $ and $ S_4 $ at fixed points of $ \iota'_0 $ as in the previous subsection. Remark that the point $ p_1 $ in $ S_0 $ that connects to $ S_1 $ and the point $ p_2 $ in $S_0$ that connects to $ S_2 $ are exchanged by the action of $ \iota'_0 $. If we define $ p_3 $ and $ p_4 $ in a similar way, they are exchanged by $ \iota_0 $. Then we can define commutative two involutions on $ S_1 \sharp S_2 \sharp S_3 $. One is equal to $ \sharp^3 \iota_0 $ defined as before on $ (\sharp_{k=0}^5 S_k) \setminus \{ p_3, p_4 \} $ and exchanges $ S_3 $ and $ S_4 $. The other is defined similarly and it is equal to $ \sharp^3 \iota'_0 $ on $ (\sharp_{k=0}^5 S_k) \setminus \{ p_1, p_2 \} $ and exchanges $ S_1 $ and $ S_2 $. Using an abuse of notations we denote these involutions by the same notations $ \sharp^3 \iota_0 $ and $ \sharp^3 \iota'_0 $ respectively. See Figure 5.

\begin{figure}[h]
	\def \svgwidth{0.7\columnwidth}
	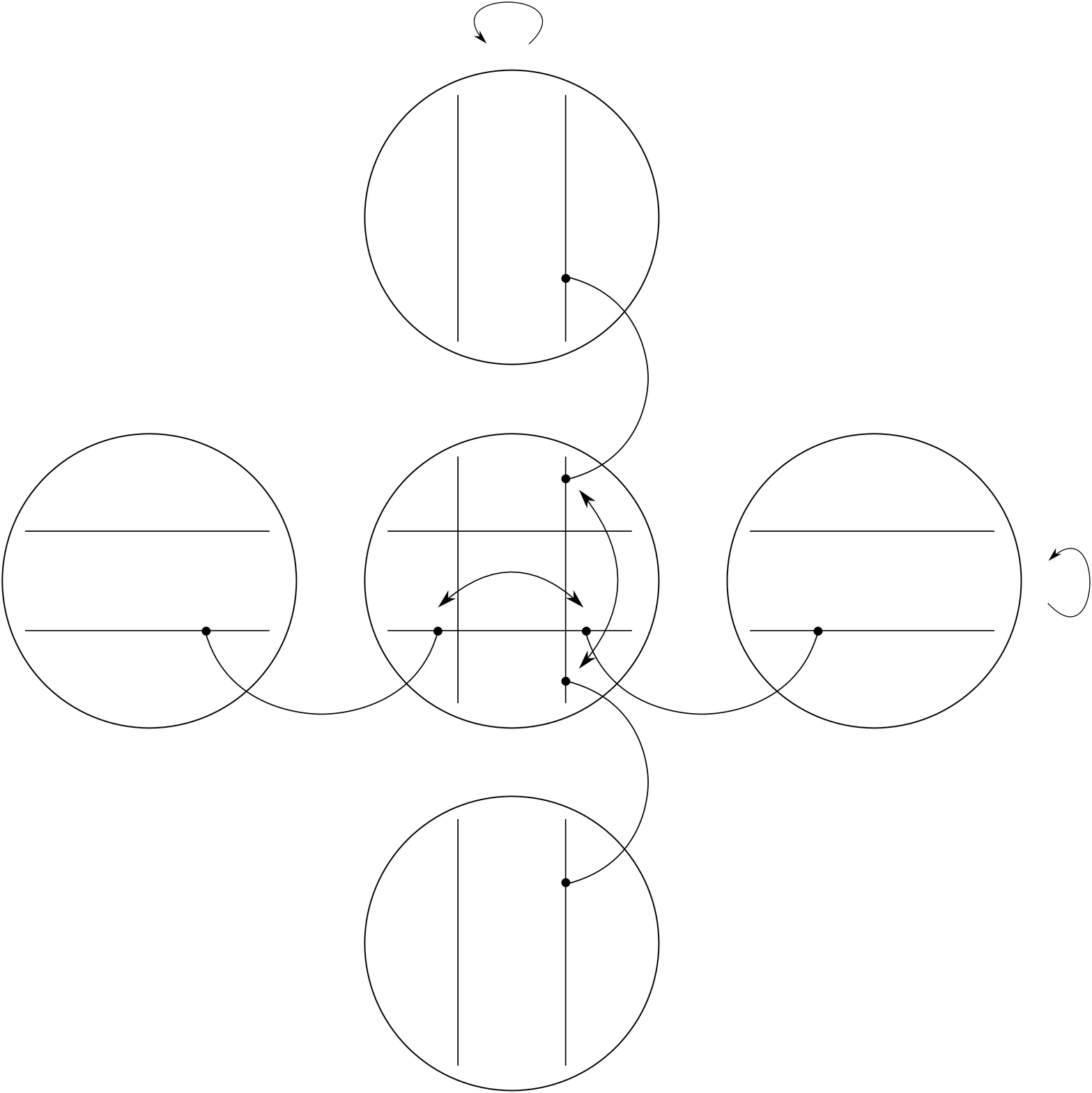
	\caption{}
\end{figure}

Similarly we take connected sum of $ 2l_1+2l_2+1 $-copies of $ S^2 \times S^2 $ and define commutative two involutions $ \sharp^{2l_1+1} \iota_0 $ and $ \sharp^{2l_2+1} \iota'_0 $. Set $Y := \sharp^{2l_1+2l_2+1}(S^2 \times S^2)$  In other words, we have a group action $ \rho : \mathbb{Z}_2 \times \mathbb{Z}_2 \to \mathrm{Diffeo}(Y) $ so that $ \rho(1,0) = \sharp^{2l_1+1} \iota_0 $ and $ \rho(0,1) = \sharp^{2l_2+1} \iota'_0 $.


The last step of the construction is to take connected sum of $ Y $ and four copies of $ k(-E_8) $. We denote four copies of $ k(-E_8) $ by $ W_1 $, $ W_2 $, $ W_3 $ and $ W_4 $ respectively and set $ Z := \coprod_{k=1}^4 W_k $. We define $ \lambda : \mathbb{Z}_2 \times \mathbb{Z}_2 \to \mathrm{Diffeo}(Z) $, an action of $ \mathbb{Z}_2 \times \mathbb{Z}_2 $ on $ Z $, as follows. The action of $ \lambda(1,0) $ exchanges $ W_1 $ and $ W_2 $, and also exchanges $ W_3 $ and $ W_4 $. Similarly the action of $ \lambda(0,1) $ exchanges $ W_1 $ and $ W_3 $, and also exchanges $ W_2 $ and $ W_4 $.

Take a point $ p \in Y $ outside the fixed point of $ \rho(1,0) $ and $ \rho(0,1) $, and also take a point $ q \in W_1 $. We modify $ Y $ as follows: remove small open balls $ B_p^1 $, $ B_p^2 $, $ B_p^3 $ and $ B_p^4 $ around $ p $, $ \rho(1,0)(p) $, $ \rho(0,1)(p) $ and $ \rho(1,1)(p) $. Similarly on $ Z $, remove small open balls $ B_q^1 $, $ B_q^2 $, $ B_q^3 $ and $ B_q^4 $ around $ q $, $ \lambda(1,0)(q) $, $ \lambda(0,1)(q) $ and $ \lambda(1,1)(q) $. Then we have two $ \mathbb{Z}_2 \times \mathbb{Z}_2 $-actions on $ \coprod_{k=1}^4 S^3 \cong \partial(Y \setminus \cup_{k=1}^4 B_p^k) \cong \partial(Z \setminus \cup_{k=1}^4 B_q^k) $ induced by $ \rho $ and $ \lambda $, and these two agree. So we can define a $ \mathbb{Z}_2 \times \mathbb{Z}_2 $-action on $ X := Y \sharp_{k=1}^4 W_k $ which is equal to $ \rho $ on $ Y \setminus \cup_{k=1}^4 B_p^k $ and $ \lambda $ on $ Z \setminus \cup_{k=1}^4 B_q^k $. We write this involution $ \sigma $ and define two commutative involutions on $ X $ as $ \iota_1 := \sigma(1,0) $ and $ \iota_2 := \sigma(0,1) $.


We show that $ \iota_1 $, $ \iota_2 $ and $ \iota_1 \circ \iota_2 $ are smoothable but $ \sigma $ is nonsmoothable for some parameters.

\begin{prp}
	$ \iota_1 $, $ \iota_2 $ and $ \iota_1 \circ \iota_2 $ are smoothable if $ l_1 \geq 3k $ and $ l_2 \geq 3k $.
\end{prp}
\begin{proof}
	We show the smoothability of $ \iota_1 $. In the construction of the action $ \sigma $, the homeomorphism class of $X$ is invariant even if we the points $p$ at which take connected sum with $ W_1 $. So we choose $ p $ outside the horizontal component of $ \sharp^{2l_1+2l_2+1} (S^2 \times S^2) $ in the Figure 6. Then the upper part of $ X $ in Figure 6 is homeomorphic to $ l_2 (S^2 \times S^2) \sharp 2k(-E_8) \setminus \mathrm{small ball} $ and it is homeomorphic to $ (l_2 - 3k)(S^2 \times S^2) \sharp k(K3) \setminus \mathrm{small ball} $ which has at least one smooth structure. a similar is applied to the lower part of $X$ in Figure 6. So $ \iota_1 $ is topologically conjugate to the connected sum of two involutions at points outside the fixed point set. One is $ \sharp^{2l_1+1} \iota_0 $ on $ \sharp^{2l_1+1} (S^2 \times S^2)$ and the other is the exchange of the two components of $ (l_2 - 3k)(S^2 \times S^2) \sharp k(K3) \coprod (l_2 - 3k)(S^2 \times S^2) \sharp k(K3) $. These two involution are smooth for some smooth structures. So $ \iota_1 $, their connected sum, is also smooth for some smooth structure on $ X $.
	
	The smoothability of $ \iota_2 $ and $ \iota_1 \circ \iota_2 $ are proved in a similar way.
\end{proof}

\begin{figure}[h]
	\def \svgwidth{\columnwidth}
	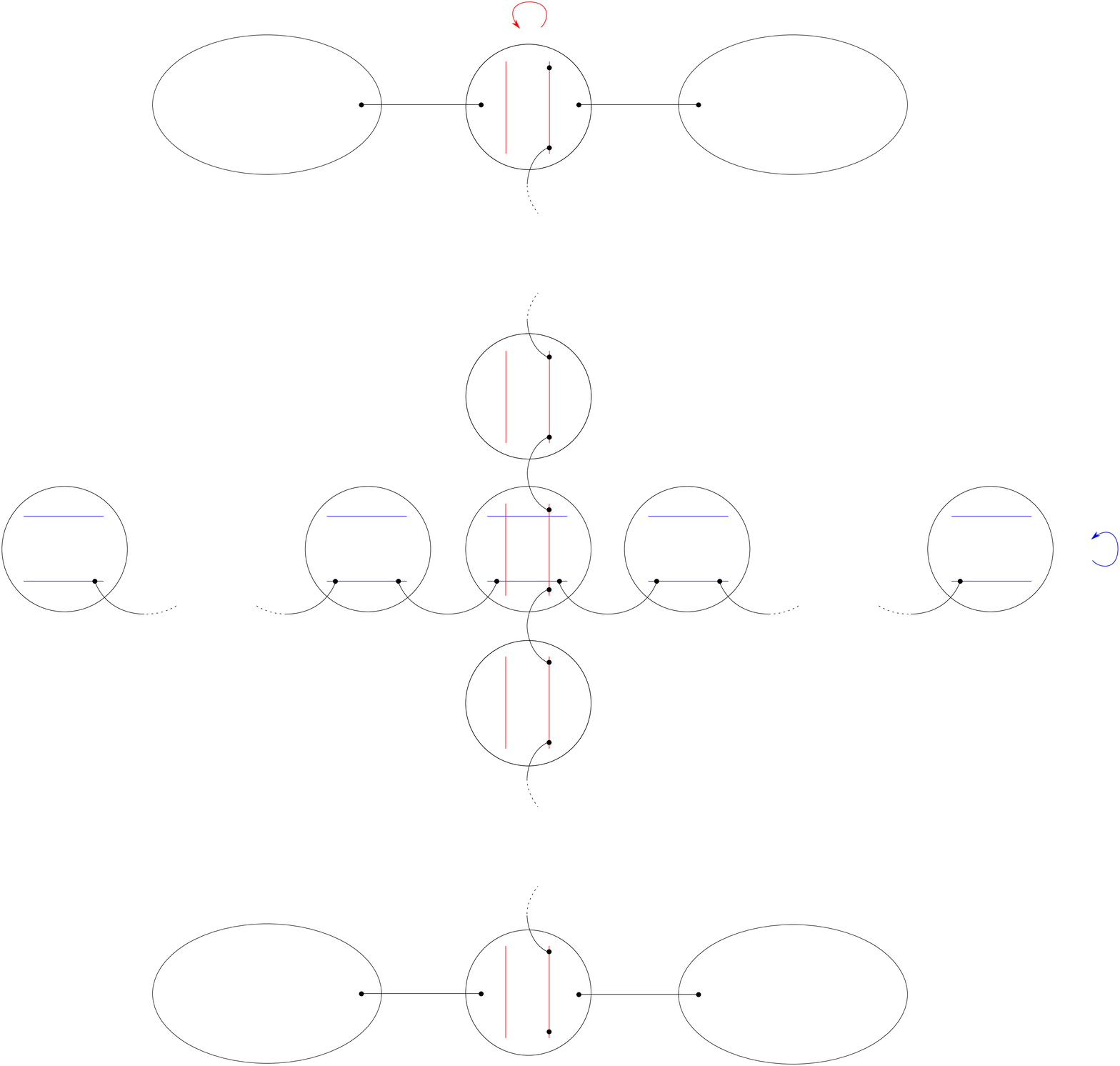
	\caption{}
\end{figure}

\begin{prp}
	$ \sigma $ is nonsmoothable if $ k \geq 1 $.
\end{prp}
\begin{proof}
	Suppose that there exists a smooth structure on $ X $ so that $ \rho $ is smooth. Since $ X $ satisfies $ b_1 = 0 $, it has a unique spin structure $ \mathfrak{s} $. Note that $ \iota_1 $ and $ \iota_2 $ are odd, smooth and preserving the spin structure $ \mathfrak{s} $. The fixed point set of $ \iota_1 \circ \iota_2 $ is $ \{ (0,0,1),(0,0,-1) \} \times \{ (0,0,1),(0,0,-1) \} $. Since this fixed point set is a 0-dimensional submanifold of $ X $, which implies that $\iota_1 \circ \iota_2$ is an even involution.
	Theorem $\ref{thm2}$ implies
	\begin{align*}
	b_+^{\langle I_1, I_2 \rangle} (X) \geq - \frac{1}{32} \sigma(X) + \frac{1}{8} \mathrm{index}_{I_1 \circ I_2} D.
	\end{align*}
On the other hand a simple calculation shows $b_+^{\langle I_1, I_2 \rangle} (X) = 0$. From Proposition $\ref{prp1}$, $\mathrm{index}_{I_1 \circ I_2} D$ is calculated by the data around the fixed points, which is reduced to the case for an involution on $S^2 \times S^2$. Since $S^2 \times S^2$ allows positive scalar curvature metric which is invariant under the involution, $\mathrm{index}_{I_1 \circ I_2} D$ is equal to $0$.
	So we obtain $ 0 \geq k $, which is a contradiction if $ k \geq 1 $.
\end{proof}


\section{An equivariant version of the Seiberg-Witten theory}

\subsection{A review of the Seiberg-Witten theory}

Let $ X $ be a closed oriented spin 4-manifold with a Riemannian metric, and $ SO(X) $ be its orthonormal frame bundle. Then, a spin structure $ \mathfrak{s} $ of $ X $ is a pair of $ Spin(4) $-bundle $ P $ and a map $ \pi : P \to SO(X) $ which commutes the following diagram.
\begin{align*}
\xymatrix{ P \ar[rr]^{\pi} \ar[rd]_{Spin(4)} & & SO(X) \ar[ld]^{SO(4)} \\
	 & X &
}
\end{align*}

Let $ (\Delta_\pm, \mathbb{H}) $ be the complex representations of $ Spin(X) $ defined by
\begin{align*}
\Delta_\pm : Spin(4) \cong Sp(1) \times Sp(1) \ni (q_+,q_-) \mapsto (\mathbb{H} \ni x \mapsto q_\pm x \in \mathbb{H}) \in GL(\mathbb{H}).
\end{align*}
where complex structure of $ \mathbb{H} $ is given by right multiplication. We call the associated vector bundles $ S^\pm := P \times_{\Delta_\pm} \mathbb{H} $ spinor bundles. $\gamma ([u, (x \otimes a, \phi_+)]) := I ([u, x \phi_+ \bar{a}])$. Let $\omega$ be the connection on $ P $ that is defined to be the pull-back of Levi-Civita connection by $ \pi : P \to SO(X) $.
Let $A$ be a connection on $P$ of the form $A = \omega + a$ for some $a \in i\Omega^1(X)$.
We denote the Dirac operator twisted with $A$ by $D_A : \Gamma(S^{+}) \to \Gamma(S^{-})$.
Let $d^* : \Omega^2 (X) \to \Omega^1 (X)$ be the formal adjoint of $d : \Omega^1(X) \to \Omega^2(X)$.
Let $d^{+} : \Omega^1 (X) \to \Omega^{+} (X)$ be $d^{+} := (d + *d)/2$ for the exterior derivative $d : \Omega^1(X) \to \Omega^2(X) $ and the Hodge star operation $*$.
We define $\sigma : \Gamma(S^{+}) \to i\Omega^{+} (X)$ by the formula
\begin{align*}
\sigma([u,\phi]) = [u,\phi i \bar{\phi}] \in P \times_{Spin(4)} \mathrm{Im} \mathbb{H} \cong i\Lambda^{+}
\end{align*}
for $[u,\phi] \in P \times_{\Delta_{+}} \mathbb{H}$.

Now we define the monopole map $SW : i\mathrm{Im}d^* \times \Gamma(S^{+}) \to i\Omega^{+} (X) \times \Gamma(S^{-})$ by the formula
\begin{align*}
SW(a,\phi) = (d^{+}a - \sigma(\phi),D_a \phi) \in i\Omega^{+} (X) \times \Gamma(S^{-})
\end{align*}
for $(a,\phi) \in  i\mathrm{Im}d^* \times \Gamma(S^{+})$.

\begin{rem}
The monopole map $SW$ defined above has a natural extension to a map refined on $i\Omega^1(X) \times \Gamma(S^{+})$, which is actually called monopole map. The above map $SW$ is a restriction of the usual monopole map on $i\mathrm{Im}d^* \times \Gamma(S^{+})$. \end{rem}

\subsection{Odd involution}\label{Odd}

Suppose $ \iota $ is an involution on $ X $ preserving the isomorphism class spin structure $ \mathfrak{s}$. Let $\iota_{*} : SO(X) \to SO(X) $ be the cannonical lift of $\iota$. Then there exists a $ Spin(4) $-equivariant map $ \tilde{\iota} : P \to P $ satisfying the following commutative diagram.
\begin{align*}
\xymatrix@!{
	P \ar[r]^{\tilde{\iota}} \ar[d]^{\pi} & P \ar[d]^{\pi} \\
	SO(X) \ar[r]^{\iota_{*}} & SO(X)
}
\end{align*}

We assume that $\iota$ is an odd involution. A sufficient condition for $\iota$ to be odd is given by the following lemma.

\begin{lem}\label{lem1}
	 If the fixed point set of $\iota$ contains 2-dimensional component,  $ \iota $ is odd. 
\end{lem}

\begin{proof}
	It suffices to show the relation $\tilde{\iota} \circ \tilde{\iota} = -1$ on a fixed point of $\iota$.
	$ \iota_* : SO(X) \to SO(X) $ is locally equivariantly diffeomorphic to the following action around a fixed point in the 2-dimensional component.
	\begin{align*}
	\mathbb{R}^4 \times SO(4) \ni \left( \begin{pmatrix} x \\ y \\ z \\ w \end{pmatrix}, A \right) \mapsto \left( \begin{pmatrix} -x \\ -y \\ z \\ w \end{pmatrix}, \begin{pmatrix} -1 & & & \\ & -1 & & \\ & & 1 & \\ & & & 1 \end{pmatrix} A \right) \in \mathbb{R}^4 \times SO(4).
	\end{align*}
	Recall that a double covering map $ Spin(4) \cong Sp(1) \times Sp(1) \to SO(4) $ is given by
	\begin{align*}
	Sp(1) \times Sp(1) \ni (q_+, q_-) \mapsto (\mathbb{H} \ni x \mapsto q_- x q_+^{-1} \in \mathbb{H}) \in SO(4)
	\end{align*}
	In this fomulation $ \tilde{\iota} $ is locally described by
	\begin{align*}
	\mathbb{R}^4 \times Spin(4) \ni \left( \begin{pmatrix} x \\ y \\ z \\ w \end{pmatrix}, (q_+, q_-) \right) \mapsto \left( \begin{pmatrix} -x \\ -y \\ z \\ w \end{pmatrix}, \pm (q_+ i, q_- i) \right) \in \mathbb{R}^4 \times Spin(4),
	\end{align*}
	which implies that $ \tilde{\iota} \circ \tilde{\iota} = -1 $ around fixed points. 
\end{proof}

We define involutions on three spaces; the space of differential forms $ i\Omega^{*}(X) $, the spinor bundles $ S^{\pm} $ and the gauge group $ \mathcal{G} := \mathrm{Map}(X,Pin(2)) $. We denote them by the same notation $I$ and define $I$ by the following formula 
\begin{align*}
I(a) &:= -\iota^{*}a \\
I([u,\phi]) &:= [\tilde{\iota}(u),\phi j^{-1}] \\
I(g) &:= j\iota^{*}gj^{-1}
\end{align*}
 for $ a \in i\Omega^{*}(X) $, $ [u,\phi] \in P \times_{\Delta_{\pm}} \mathbb{H} = S^{\pm} $ and $ g \in \mathcal{G} $.
 
Since $\tilde{\iota} \circ \tilde{\iota} = -1$ and $j ^2 = -1$, $I$ is an involution on the spinors. The involutions $ I : S^{\pm} \to S^{\pm} $ induces involutions on the space of the sections of spinor bundles $ \Gamma(S^{\pm}) $, we also denote these involutions by $ I : \Gamma(S^{\pm}) \to \Gamma(S^{\pm})$.

\begin{lem}
The Clifford multiplication $ \gamma $, the quadratic map $ \sigma $, gauge action, and the Dirac operator $ D $ are $ I $-equivariant.
\end{lem}

\begin{proof}
	The equivariance of $ \gamma : \Lambda^1_\mathbb{C} \times S^+ \to S^- $ is shown in the following way Let $ [u, (x \otimes a,\phi_+)] $ be an element of $ P \times_{Spin(4)} (\mathbb{H}_T \otimes \mathbb{C}, \mathbb{H}_+) \cong \Lambda^1_\mathbb{C} \times S^+ $. Then we have
	\begin{align*}
	I \gamma ([u, (x \otimes a, \phi_+)]) &= I ([u, x \phi_+ \bar{a}]) = [\tilde{\iota}(u), x \phi_+ \bar{a} j^{-1}], \\
	 \gamma I ([u, (x \otimes a, \phi_+)]) &= \gamma ([\tilde{\iota}(u), (x \otimes \bar{a}, \phi_+ j^{-1})]) \\ &= [\tilde{\iota}(u), x \phi_+ j^{-1} a ] = [\tilde{\iota}(u), x \phi_+ \bar{a} j^{-1}].
	\end{align*}
Similarly we can check the equivariance of $\sigma$ and $D$ straightforwardly.
\end{proof}

\begin{cor}
	the following diagram is commutative and $ I $-equivariant.
	\begin{align*}
	\xymatrix{
		Pin(2) \times ( i\mathrm{Im(d^{*})} \times \Gamma (S^{+}) ) \ar[r]^{\mathrm{id} \times SW} \ar[d]^{\mbox{\rm{gauge action}}} & Pin(2) \times ( i\Omega^{+} (X) \times \Gamma (S^{-}) ) \ar[d]^{\mbox{\rm{gauge action}}} \\
		i\mathrm{Im(d^{*})} \times \Gamma (S^{+}) \ar[r]^{SW} & i\Omega^{+} (X) \times \Gamma (S^{-})
	}
	\end{align*}
\end{cor}

It implies that the following diagram is also commutative.
\begin{align*}
\xymatrix{
	Pin(2)^I \times ( i\mathrm{Im(d^{*})}^I \times \Gamma (S^{+})^I ) \ar[r]^{\mathrm{id} \times SW^I} \ar[d]^{\mbox{gauge action}} & Pin(2)^I \times ( i\Omega^{+} (X)^I \times \Gamma (S^{-})^I ) \ar[d]^{\mbox{gauge action}} \\
	i\mathrm{Im(d^{*})}^I \times \Gamma (S^{+})^I \ar[r]^{SW^I} & i\Omega^{+} (X)^I \times \Gamma (S^{-})^I
}
\end{align*}
where the $ I $-indexed objects are $ I $-invariant part. And this means that $ SW^I $, the restriction of the monopole map $ SW $ to the $ I $-invariant part, is $ Pin(2)^I $-equivariant. And by the definition of the involution $ I $ on the gauge group, $ Pin(2)^I = \langle j \rangle $ and this is abstractly isomorphic to $ \mathbb{Z}_4 $ as group. So we have an $ \mathbb{Z}_4 $-equivariant map
\begin{align*}
SW^I : i\mathrm{Im(d^{*})}^I \times \Gamma (S^{+})^I \to i\Omega^{+} (X)^I \times \Gamma (S^{-})^I.
\end{align*}
We call this $ I $-invariant monopole map.

\section{Proof of Main Theorems}

In this section we prove Theorem $ \ref{thm1} $ and Theorem $ \ref{thm2} $. We use the equivariant $ K $-theory and a finite dimensional approximation technique introduced in \cite{F1}. 
\subsection{Equivariant K-theory}

We review several facts on equivariant K-theory, especially, the equivariant Thom isomorphism and tom Dieck’s character formula for the K-theoretic degree.

Let $ V $ and $ W $ be complex $G$ representations for some compact Lie group $G$. Let $BV$ and $BW$ denote closed balls in $V$ and $W$ and $f:BV \to BW$ are a $G$-map presering the boundaries $SV$ and $SW$. By the equivariant Thom isomorphism theorem, $K_G(BV,SV)$ is a free $R(G)$-module generated by the Bott class $\lambda_V$. $f$ induces a map $ f^* : K_G(BW,SW) \to K_G(BV,SV) $ which defines an unique element $\alpha_f \in R(G)$ by $f^*(\lambda_V) = \alpha_f \cdot \lambda_W$.

Let $V_g$ and $W_g$ be the subspace of $V$ and $W$ fixed by an elemant $g \in G$ and let $V_g^\perp$ and $W_g$ be the orthogonal complements. Let $ f_g : V_g \to W_g $ denote the restriction of $f$ and $d(f_g)$ be the ordinal topological degree of $f_g$. For any $\beta \in R(G)$ $\lambda_{-1}$ denote $\Sigma (-1)^i \Lambda^i \beta$.

According to Bryan \cite{Br} we use the following formula.

\begin{lem}[tom Dieck]\label{lem2}
$\mathrm{tr}_g(\alpha_f) = d(f_g)\mathrm{tr}_g(\lambda_{-1}(W_g^\perp - V_g^\perp))$.
\end{lem}

\subsection{Proof of Theorem 2.3}

In the previous section, we constructed an $ \mathbb{Z}_4 $-equivariant map
\begin{align*}
\ell + c : \mathcal{V} \to \mathcal{W}
\end{align*}
for $ \mathcal{V} = i\mathrm{Im(d^{*})}^I \times \Gamma (S^{+})^I $ and $ \mathcal{W} = i\Omega^{+} (X)^I \times \Gamma (S^{-})^I $ where $ \ell = (\mathrm{d}^{+},D) $ is the linear part of $ I $-invariant monopole map $SW^I$ and $ c $ is the quadratic part of $SW^I$ so that $SW^I = \ell + c$.

The space $ SW^{-1} (0) $ is the moduli space of the solutions to Seiberg-Witten equations. Since we assume that $b_1(X) = 0$, $SW^{-1}(0)$ is compact, which implies $ (SW^I)^{-1} (0) $ is also compact. Take a large $R \geq 0$ such that the $L^2_3$-ball $B(R)$ with radius $R$ centred in the origin contains $SW^{-1}(0)$ in its interior. 

We take a sufficiently large finite dimensional subspace $W$ of $\mathcal{W}$ containing cokernel of $\ell$ which we identify with the $L^2$-orthogonal complement of the image of $\ell$. Let $V$ be the inverse image of $ W $ by $ \ell $. Let $\psi$ be composition of the restriction of the $ SW^I $ to $ V $ and the $L^2$-projection $ \mathrm{pr}_W : \mathcal{W} \to W $ i.e. $ \psi = \mathrm{pr}_W \circ SW^I | _V $. The argument in \cite{F1} implies ;

\begin{lem}$\mathrm{(Furuta}$\cite{F1}$\mathrm{)}$
	There exists a finite dimensional subspace $ W_0 \subset \mathcal{W} $ such that for any finite dimensional subspace $ W \subset \mathcal{W} $ which contains $ W_0 $, the intersection of the ball $B(R)$ and the inverse image of $ 0 $ by $ \psi : V \to W $ is compact.
\end{lem}

So we have a $ \mathbb{Z}_4 $-equivariant map $ \psi : V \to W $ such that the inverse image of $ 0 $ is compact.
$ V $ and $ W $ are finite dimensional representations of $ \mathbb{Z}_4 $. By definition of the gauge action, there exists the following isomorphisms as real representations of $ \mathbb{Z}_4 $;
\begin{align*}
V \cong \tilde{\mathbb{R}}^m \oplus \mathbb{C}_1^{n+k} \quad \mbox{\rm{and}} \quad W \cong \tilde{\mathbb{R}}^{m+b} \oplus \mathbb{C}_1^n,
\end{align*}
where $ \mathbb{Z}_4 $ acts on $ \tilde{\mathbb{R}} $ by multiplication of $ \pm 1 $ via the surjection $ \mathbb{Z}_4 \to \{ \pm 1 \} $, and on $ \mathbb{C}_a $ by multiplication via $ \mathbb{Z}_4 \ni a \mapsto i^a \in \mathbb{C}^{*} $ for $ a \in \mathbb{Z} $, $ m $, $ n $ are some nonnegative integer, $ b = b_+^I (X) $, and $ k = \frac{1}{2} \mathrm{index}_\mathbb{R} (D^I)$. Note that $k$ is explicitly calculated as follows. Since the action of $ I $ on $ \Gamma (S^{\pm}) $ is anti-complex linear and the Dirac operator $ D $ is $ I $-equivariant, $ I $ also acts on $ \mathrm{ker} D $ and $ \mathrm{coker} D $ as an anti-complex linear map satisfying $ I^2 = \mathrm{id} $. In other words, $ I $ is a real structure on $ \mathrm{ker} D $ and $ \mathrm{coker} D $, which implies
\begin{align*}
\mathrm{dim}_\mathbb{R} ( \mathrm{ker} (D^I) ) = \mathrm{dim}_\mathbb{R} ( ( \mathrm{ker} D)^I ) = \frac{1}{2} \mathrm{dim}_\mathbb{R} ( \mathrm{ker} D ) = \mathrm{dim}_\mathbb{C} ( \mathrm{ker} D )
\end{align*}
and similarly we have $\mathrm{dim}_\mathbb{R} ( \mathrm{coker} (D^I) ) = \mathrm{dim}_\mathbb{C} ( \mathrm{ker} (D) )$.

In particular we have
\begin{align*}
2\mathrm{index}_\mathbb{R} (D^I) = \mathrm{index}_\mathbb{R} D.
\end{align*}
So we obtain
\begin{align*}
k = \frac{1}{2} \mathrm{index}_\mathbb{R} (D^I) = \frac{1}{4} \mathrm{index}_\mathbb{R} D = -\frac{1}{16} \sigma (X).
\end{align*}

We have a $ \mathbb{Z}_4 $-equivariant map $ \psi_\mathbb{C} : V_\mathbb{C} \to W_\mathbb{C} $ by taking complexification of $ \psi $ as $ \psi_\mathbb{C} (u \otimes 1 + v \otimes i ) := \psi(u) \otimes 1 + \psi(v) \otimes i $. The intersection of the ball $B$ with radius $R$ and the inverse image of $ \psi_\mathbb{C} $ is also compact. So by the same technique in \cite{F1}, we have a $ \mathbb{Z}_4 $-equivariant map $ f $ from a $ \mathbb{Z}_4 $-invariant ball $ BV_\mathbb{C} $ in $ V $ to a $ \mathbb{Z}_4 $-invariant ball $ BW_\mathbb{C} $ in $ W $ preserving their boundaries $ SV_\mathbb{C} $ and $ SW_\mathbb{C} $.

Now we apply tom Dieck's formula to $ f : BV_\mathbb{C} \to BW_\mathbb{C} $. As complex representations of $ \mathbb{Z}_4 $, $ V_\mathbb{C} $ is isomorphic to $ \mathbb{C}_2^m \oplus ( \mathbb{C}_1 \oplus \mathbb{C}_{-1} )^{n+k} $ and $ W_\mathbb{C} $ to $ \mathbb{C}_2^{m+b} \oplus ( \mathbb{C}_1 \oplus \mathbb{C}_{-1} ) $, so we have
\begin{align*}
\mathrm{tr}_j ( \mathrm{deg} f ) = \mathrm{tr}_j ( \Lambda_{-1} ( \mathbb{C}_2^b - ( \mathbb{C}_1 \oplus \mathbb{C}_{-1} )^k ) ) = 2^{b-k}.
\end{align*}
Since $ \mathrm{tr}_j ( \mathrm{deg} f ) $ must be an algebraic integer, we have $ b \geq k $, which implies the inequality

\begin{align*}
b_+^I (X) = b \geq k = -\frac{1}{16} \sigma (X).
\end{align*}

\subsection{Proof of Theorem $ \ref{thm2} $}
We can show theorem $\ref{thm2}$ by an argument parallel to the proof of theorem $\ref{thm1}$. The only difference is we use the following $b$ anf $k$. When we apply Lemma \ref{lem2};
\begin{align*}
b&=b_{+}^{\langle I_1, I_2 \rangle}(X) \\
k&= \frac{1}{2} \mathrm{index}_\mathbb{R}D^{\langle I_1, I_2 \rangle} \\
 &= \frac{1}{2} \frac{\mathrm{index}_\mathrm{id}D + \mathrm{index}_{I_1}D + \mathrm{index}_{I_2}D + \mathrm{index}_{I_1 \circ I_2}D}{4}.
\end{align*}
Note that we have
\begin{align*}
\mathrm{index}_{I_1} D = 2\mathrm{index}D^{I_1} - \mathrm{index}D = 0.
\end{align*}
Similarly we have $\mathrm{index}_{I_2} D = 0$.
Hence we obtain the inequality:
\begin{align*}
b_+^{\langle I_1, I_2 \rangle} (X) = b \geq k = \frac{1}{2} \frac{-\sigma(X)/4 + \mathrm{index}_{I_1 \circ I_2} D}{4}.
\end{align*}
Similarly if we take another lift $I_1$ of $\iota_1$, we obtain
\begin{align*}
b_+^{\langle I_1, I_2 \rangle} (X) = b \geq k = \frac{1}{2} \frac{-\sigma(X)/4 - \mathrm{index}_{I_1 \circ I_2} D}{4}.
\end{align*}
So we have the required inequality.

\end{document}

%% file: s2s2.eps_tex
\begingroup%
  \makeatletter%
  \providecommand\color[2][]{%
    \errmessage{(Inkscape) Color is used for the text in Inkscape, but the package 'color.sty' is not loaded}%
    \renewcommand\color[2][]{}%
  }%
  \providecommand\transparent[1]{%
    \errmessage{(Inkscape) Transparency is used (non-zero) for the text in Inkscape, but the package 'transparent.sty' is not loaded}%
    \renewcommand\transparent[1]{}%
  }%
  \providecommand\rotatebox[2]{#2}%
  \ifx\svgwidth\undefined%
    \setlength{\unitlength}{808.5421627bp}%
    \ifx\svgscale\undefined%
      \relax%
    \else%
      \setlength{\unitlength}{\unitlength * \real{\svgscale}}%
    \fi%
  \else%
    \setlength{\unitlength}{\svgwidth}%
  \fi%
  \global\let\svgwidth\undefined%
  \global\let\svgscale\undefined%
  \makeatother%
  \begin{picture}(1,1.20980883)%
    \put(0,0){\includegraphics[width=\unitlength]{s2s2.eps}}%
    \put(0.12537557,0.8891647){\color[rgb]{0,0,0}\makebox(0,0)[lt]{\begin{minipage}{0\unitlength}\raggedright \end{minipage}}}%
    \put(0.36175987,0.57807848){\color[rgb]{0,0,0}\makebox(0,0)[lb]{\smash{$S^2 \times S^2$}}}%
    \put(0.06277355,1.18734864){\color[rgb]{0,0,0}\makebox(0,0)[lb]{\smash{}}}%
    \put(-0.13312461,1.35126342){\color[rgb]{0,0,0}\makebox(0,0)[lt]{\begin{minipage}{0.80558114\unitlength}\raggedright \end{minipage}}}%
    \put(0.18471015,1.17935278){\color[rgb]{0,0,0}\makebox(0,0)[lb]{\smash{$ \{ (1,0,0) \} \times S^2 $}}}%
    \put(0.18670912,0.00796186){\color[rgb]{0,0,0}\makebox(0,0)[lb]{\smash{$ \{ (-1,0,0) \} \times S^2 $}}}%
  \end{picture}%
\endgroup%

%% file: r4r4.eps_tex
\begingroup%
  \makeatletter%
  \providecommand\color[2][]{%
    \errmessage{(Inkscape) Color is used for the text in Inkscape, but the package 'color.sty' is not loaded}%
    \renewcommand\color[2][]{}%
  }%
  \providecommand\transparent[1]{%
    \errmessage{(Inkscape) Transparency is used (non-zero) for the text in Inkscape, but the package 'transparent.sty' is not loaded}%
    \renewcommand\transparent[1]{}%
  }%
  \providecommand\rotatebox[2]{#2}%
  \ifx\svgwidth\undefined%
    \setlength{\unitlength}{5917.76989896bp}%
    \ifx\svgscale\undefined%
      \relax%
    \else%
      \setlength{\unitlength}{\unitlength * \real{\svgscale}}%
    \fi%
  \else%
    \setlength{\unitlength}{\svgwidth}%
  \fi%
  \global\let\svgwidth\undefined%
  \global\let\svgscale\undefined%
  \makeatother%
  \begin{picture}(1,0.42566077)%
    \put(0,0){\includegraphics[width=\unitlength]{r4r4.eps}}%
    \put(0.20127614,0.31489631){\color[rgb]{0,0,0}\makebox(0,0)[lb]{\smash{}}}%
    \put(0.0601138,0.28380873){\color[rgb]{0,0,0}\makebox(0,0)[lb]{\smash{$\{ (0,0) \} \times \mathbb{R}^2$}}}%
    \put(0.02348499,0.36892473){\color[rgb]{0,0,0}\makebox(0,0)[lb]{\smash{$\mathbb{R}^4$}}}%
    \put(0.0114319,0.05373585){\color[rgb]{0,0,0}\makebox(0,0)[lb]{\smash{$(0,0,0,0)$}}}%
    \put(0.57771675,0.28206694){\color[rgb]{0,0,0}\makebox(0,0)[lb]{\smash{$\{ (0,0) \} \times \mathbb{R}^2$}}}%
    \put(0.53780808,0.36917842){\color[rgb]{0,0,0}\makebox(0,0)[lb]{\smash{$\mathbb{R}^4$}}}%
    \put(0.67535944,0.0521266){\color[rgb]{0,0,0}\makebox(0,0)[lb]{\smash{$(0,0,0,0)$}}}%
    \put(0.9088947,0.26744686){\color[rgb]{0,0,0}\makebox(0,0)[lb]{\smash{$ i_0 \sharp i_0$}}}%
  \end{picture}%
\endgroup%

%% file: conn_sum_1.eps_tex
\begingroup%
  \makeatletter%
  \providecommand\color[2][]{%
    \errmessage{(Inkscape) Color is used for the text in Inkscape, but the package 'color.sty' is not loaded}%
    \renewcommand\color[2][]{}%
  }%
  \providecommand\transparent[1]{%
    \errmessage{(Inkscape) Transparency is used (non-zero) for the text in Inkscape, but the package 'transparent.sty' is not loaded}%
    \renewcommand\transparent[1]{}%
  }%
  \providecommand\rotatebox[2]{#2}%
  \ifx\svgwidth\undefined%
    \setlength{\unitlength}{10040.62097533bp}%
    \ifx\svgscale\undefined%
      \relax%
    \else%
      \setlength{\unitlength}{\unitlength * \real{\svgscale}}%
    \fi%
  \else%
    \setlength{\unitlength}{\svgwidth}%
  \fi%
  \global\let\svgwidth\undefined%
  \global\let\svgscale\undefined%
  \makeatother%
  \begin{picture}(1,0.19624354)%
    \put(0,0){\includegraphics[width=\unitlength]{conn_sum_1.eps}}%
    \put(0.55008871,0.08774548){\color[rgb]{0,0,0}\makebox(0,0)[lb]{\smash{$\cdots$}}}%
    \put(0.0549691,0.09287886){\color[rgb]{0,0,0}\makebox(0,0)[lb]{\smash{$S^2 \times S^2$}}}%
    \put(0.30361641,0.09333416){\color[rgb]{0,0,0}\makebox(0,0)[lb]{\smash{$S^2 \times S^2$}}}%
    \put(0.76241079,0.09287886){\color[rgb]{0,0,0}\makebox(0,0)[lb]{\smash{$S^2 \times S^2$}}}%
    \put(0.96413474,0.13512908){\color[rgb]{0,0,0}\makebox(0,0)[lb]{\smash{$ \sharp^l \iota_0$}}}%
  \end{picture}%
\endgroup%

%% file: n_s_1.eps_tex
\begingroup%
  \makeatletter%
  \providecommand\color[2][]{%
    \errmessage{(Inkscape) Color is used for the text in Inkscape, but the package 'color.sty' is not loaded}%
    \renewcommand\color[2][]{}%
  }%
  \providecommand\transparent[1]{%
    \errmessage{(Inkscape) Transparency is used (non-zero) for the text in Inkscape, but the package 'transparent.sty' is not loaded}%
    \renewcommand\transparent[1]{}%
  }%
  \providecommand\rotatebox[2]{#2}%
  \ifx\svgwidth\undefined%
    \setlength{\unitlength}{5795.39451522bp}%
    \ifx\svgscale\undefined%
      \relax%
    \else%
      \setlength{\unitlength}{\unitlength * \real{\svgscale}}%
    \fi%
  \else%
    \setlength{\unitlength}{\svgwidth}%
  \fi%
  \global\let\svgwidth\undefined%
  \global\let\svgscale\undefined%
  \makeatother%
  \begin{picture}(1,0.64147216)%
    \put(0,0){\includegraphics[width=\unitlength]{n_s_1.eps}}%
    \put(0.55008876,0.31022897){\color[rgb]{0,0,0}\makebox(0,0)[lb]{\smash{$\cdots$}}}%
    \put(0.04388009,0.3126306){\color[rgb]{0,0,0}\makebox(0,0)[lb]{\smash{$S^2 \times S^2$}}}%
    \put(0.28934255,0.3126306){\color[rgb]{0,0,0}\makebox(0,0)[lb]{\smash{$S^2 \times S^2$}}}%
    \put(0.75196304,0.3126306){\color[rgb]{0,0,0}\makebox(0,0)[lb]{\smash{$S^2 \times S^2$}}}%
    \put(0.293716,0.55219201){\color[rgb]{0,0,0}\makebox(0,0)[lb]{\smash{$ \sharp^k (-E_8)$}}}%
    \put(0.29431643,0.06886639){\color[rgb]{0,0,0}\makebox(0,0)[lb]{\smash{$ \sharp^k (-E_8)$}}}%
    \put(0.35495728,0.39308482){\color[rgb]{0,0,0}\makebox(0,0)[lb]{\smash{$p$}}}%
    \put(0.3537565,0.23697965){\color[rgb]{0,0,0}\makebox(0,0)[lb]{\smash{$ \sharp^l \iota_0 (p)$}}}%
    \put(0.35255566,0.48974992){\color[rgb]{0,0,0}\makebox(0,0)[lb]{\smash{$q$}}}%
    \put(0.35255566,0.14091491){\color[rgb]{0,0,0}\makebox(0,0)[lb]{\smash{$ \kappa (q) $}}}%
    \put(0.97015424,0.35608569){\color[rgb]{0,0,0}\makebox(0,0)[lb]{\smash{$\iota$}}}%
  \end{picture}%
\endgroup%

%% file: 5_s2s2.eps_tex
\begingroup%
  \makeatletter%
  \providecommand\color[2][]{%
    \errmessage{(Inkscape) Color is used for the text in Inkscape, but the package 'color.sty' is not loaded}%
    \renewcommand\color[2][]{}%
  }%
  \providecommand\transparent[1]{%
    \errmessage{(Inkscape) Transparency is used (non-zero) for the text in Inkscape, but the package 'transparent.sty' is not loaded}%
    \renewcommand\transparent[1]{}%
  }%
  \providecommand\rotatebox[2]{#2}%
  \ifx\svgwidth\undefined%
    \setlength{\unitlength}{3198.38242006bp}%
    \ifx\svgscale\undefined%
      \relax%
    \else%
      \setlength{\unitlength}{\unitlength * \real{\svgscale}}%
    \fi%
  \else%
    \setlength{\unitlength}{\svgwidth}%
  \fi%
  \global\let\svgwidth\undefined%
  \global\let\svgscale\undefined%
  \makeatother%
  \begin{picture}(1,1.00046223)%
    \put(0,0){\includegraphics[width=\unitlength]{5_s2s2.eps}}%
    \put(0.36699892,0.53517223){\color[rgb]{0,0,0}\makebox(0,0)[lb]{\smash{$S_0$}}}%
    \put(0.10752205,0.46384405){\color[rgb]{0,0,0}\makebox(0,0)[lb]{\smash{$S_1$}}}%
    \put(0.77954691,0.46384405){\color[rgb]{0,0,0}\makebox(0,0)[lb]{\smash{$S_2$}}}%
    \put(0.44423016,0.79846513){\color[rgb]{0,0,0}\makebox(0,0)[lb]{\smash{$S_3$}}}%
    \put(0.44423016,0.13229095){\color[rgb]{0,0,0}\makebox(0,0)[lb]{\smash{$S_4$}}}%
    \put(0.37908712,0.43073866){\color[rgb]{0,0,0}\makebox(0,0)[lb]{\smash{$p_1$}}}%
    \put(0.51436473,0.43073866){\color[rgb]{0,0,0}\makebox(0,0)[lb]{\smash{$p_2$}}}%
    \put(0.47451933,0.37121651){\color[rgb]{0,0,0}\makebox(0,0)[lb]{\smash{$p_4$}}}%
    \put(0.47451933,0.55814557){\color[rgb]{0,0,0}\makebox(0,0)[lb]{\smash{$p_3$}}}%
    \put(0.43024665,0.48337394){\color[rgb]{0,0,0}\makebox(0,0)[lb]{\smash{$\iota'_0$}}}%
    \put(0.57142733,0.46517295){\color[rgb]{0,0,0}\makebox(0,0)[lb]{\smash{$\iota_0$}}}%
    \put(0.95550621,0.50894674){\color[rgb]{0,0,0}\makebox(0,0)[lb]{\smash{$\sharp^3 \iota_0$}}}%
    \put(0.50576102,0.97738925){\color[rgb]{0,0,0}\makebox(0,0)[lb]{\smash{$\sharp^3 \iota'_0$}}}%
  \end{picture}%
\endgroup%

%% file: n_s_2.eps_tex
\begingroup%
  \makeatletter%
  \providecommand\color[2][]{%
    \errmessage{(Inkscape) Color is used for the text in Inkscape, but the package 'color.sty' is not loaded}%
    \renewcommand\color[2][]{}%
  }%
  \providecommand\transparent[1]{%
    \errmessage{(Inkscape) Transparency is used (non-zero) for the text in Inkscape, but the package 'transparent.sty' is not loaded}%
    \renewcommand\transparent[1]{}%
  }%
  \providecommand\rotatebox[2]{#2}%
  \ifx\svgwidth\undefined%
    \setlength{\unitlength}{3256.71834536bp}%
    \ifx\svgscale\undefined%
      \relax%
    \else%
      \setlength{\unitlength}{\unitlength * \real{\svgscale}}%
    \fi%
  \else%
    \setlength{\unitlength}{\svgwidth}%
  \fi%
  \global\let\svgwidth\undefined%
  \global\let\svgscale\undefined%
  \makeatother%
  \begin{picture}(1,0.95151223)%
    \put(0,0){\includegraphics[width=\unitlength]{n_s_2.eps}}%
    \put(0.73673955,0.4542908){\color[rgb]{0,0,0}\makebox(0,0)[lb]{\smash{$\cdots$}}}%
    \put(0.95277071,0.48622477){\color[rgb]{0,0,1}\makebox(0,0)[lb]{\smash{$\iota_1$}}}%
    \put(0.18537251,0.4542908){\color[rgb]{0,0,0}\makebox(0,0)[lb]{\smash{$\cdots$}}}%
    \put(0.4969063,0.93321454){\color[rgb]{1,0,0}\makebox(0,0)[lb]{\smash{$\iota_2$}}}%
    \put(0.46730584,0.71986403){\color[rgb]{0,0,0}\makebox(0,0)[lb]{\smash{$\vdots$}}}%
    \put(0.46098923,0.18855377){\color[rgb]{0,0,0}\makebox(0,0)[lb]{\smash{$\vdots$}}}%
    \put(0.18210453,0.8564485){\color[rgb]{0,0,0}\makebox(0,0)[lb]{\smash{$\sharp^k (-E_8)$}}}%
    \put(0.66077519,0.85715033){\color[rgb]{0,0,0}\makebox(0,0)[lb]{\smash{$\sharp^k (-E_8)$}}}%
    \put(0.18421007,0.05985344){\color[rgb]{0,0,0}\makebox(0,0)[lb]{\smash{$\sharp^k (-E_8)$}}}%
    \put(0.66358258,0.05844973){\color[rgb]{0,0,0}\makebox(0,0)[lb]{\smash{$\sharp^k (-E_8)$}}}%
    \put(0.48582832,0.50271817){\color[rgb]{0,0,0}\makebox(0,0)[lb]{\smash{}}}%
  \end{picture}%
\endgroup%